\documentclass[reqno,12pt]{amsart}

\usepackage{amssymb,latexsym}

\usepackage{enumerate}

\makeatletter

\@namedef{subjclassname@2010}{

  \textup{2010} Mathematics Subject Classification}

\makeatother
\newtheorem{thm}{Theorem}[section]
\newtheorem{Thm}{Theorem}

\newtheorem{lem}[thm]{Lemma}
\newtheorem{pro}[thm]{Proposition}
\theoremstyle{definition}

\numberwithin{equation}{section}

\newcommand{\Z}{\ensuremath{\mathbb{Z}}}

\newsymbol\dnd 232D

\renewcommand{\geq}{\geqslant}
\renewcommand{\leq}{\leqslant}
\renewcommand{\mod}[1]{{\ifmmode\text{\rm\ (mod~$#1$)}\else\discretionary{}{}{\hbox{ }}\rm(mod~$#1$)\fi}}

\begin{document}

\baselineskip=17pt

\title{Large even order character sums}

\author{Leo Goldmakher}
\address{
Department of Mathematics,
University of Toronto,
40 St. George Street,
Toronto, ON,
M5S 2E4,
Canada}
\email{lgoldmak@math.toronto.edu}

\author{Youness Lamzouri}
\address{
Department of Mathematics,
University of Illinois at Urbana-Champaign,
1409 W. Green Street,
Urbana, IL,
61801
USA}
\email{lamzouri@math.uiuc.edu}

\date{}

\begin{abstract}
A classical theorem of Paley asserts the existence of an infinite family of quadratic characters whose character sums become exceptionally large. In this paper, we establish an analogous result for characters of any fixed even order. Previously our bounds were only known under the assumption of the Generalized Riemann Hypothesis.
\end{abstract}

\thanks{The first author is partially supported by an NSERC Discovery Grant.}

\maketitle

\section{Introduction}

Dirichlet characters and sums involving them have a long history stretching back to Gauss. One specific quantity which has received a lot of attention during the past century is
\[
M(\chi) := \max_{t \leq q} \left|\sum_{n \leq t} \chi(n)\right| ,
\]
where $\chi\mod{q}$ is a nonprincipal Dirichlet character. Interest in this object began with the discovery (independently made by P\'{o}lya and Vinogradov in 1918) that
${M(\chi) \ll \sqrt{q} \log q}$, an upper bound which remains the strongest known outside of special cases. In 1977, Montgomery and Vaughan \cite{MV} showed that the stronger upper bound
\begin{equation}
\label{eq:MontVau}
M(\chi) \ll \sqrt{q} \log \log q
\end{equation}
follows from the Generalized Riemann Hypothesis GRH. The goal of this note is to prove (unconditionally) that (\ref{eq:MontVau}) is best-possible for characters of any fixed even order. Precisely:
\begin{Thm}
\label{MainThm}
Let $g \geq 2$ be a fixed even integer. Then there exist arbitrarily large $q$ and primitive characters $\chi\mod{q}$ of order $g$ satisfying
\begin{equation}
\label{eq:MainThm}
M(\chi) \gg_g \sqrt{q} \log \log q .
\end{equation}
\end{Thm}

Before describing our proof of this theorem, we briefly survey the history of lower bounds on $M(\chi)$. A classical result (apparently due to Schur, according to \cite{Da}) is that $M(\chi) \gg \sqrt{q}$ for all primitive characters $\chi\mod{q}$. However, there exist some characters for which more can be said. This was first observed by Paley \cite{Pa}, who in 1932 constructed an infinite family of quadratic characters $\chi\mod{q}$ satisfying (\ref{eq:MainThm}). No other unconditional lower bounds were proved until quite recently, when the authors \cite{GL} established that there are arbitrarily large $q$ and primitive characters $\chi\mod{q}$ of fixed odd order $g$ such that
\[
M(\chi) \gg_{g,\epsilon} \sqrt{q} (\log \log q)^{\frac{g}{\pi} \sin \frac{\pi}{g} - \epsilon} .
\]
This is presumably optimal, in view of an upper bound (conditional on the GRH) of the same form proved by the first author in \cite{Go}. \\

Assuming the GRH, Granville and Soundararajan \cite{GS} proved that Paley's omega result can be extended to characters of any fixed even order. Our proof of Theorem \ref{MainThm} removes the assumption of the GRH. The argument is inspired by that of Granville and Soundararajan, but also uses elements from Paley's paper and our own previous work, as well as some new ideas. The necessary ingredients are collected in the next section. In the final section, we put them together and prove our main theorem.

\section{Ingredients}

We first recall a few standard pieces of notation. We have already used Vinogradov's notation $f(x) \ll_a g(x)$; this simply means $|f(x)|\leq Cg(x)$, where $C$ is a constant that depends only on the subscript $a$. We will also write $f(x) \asymp g(x)$ if $f(x)\ll g(x)$ and $g(x)\ll f(x)$. The normalized complex exponential
${e(x) := e^{2\pi i x}}$ will be used frequently in our arguments. In particular, it appears in the definition of the Gauss sum for a character $\chi\mod{q}$:
\[
 \tau(\chi) := \sum_{n \leq q} \chi(n) e\Big(\frac{n}{q}\Big) .
\]
Recall that $|\tau(\chi)| = \sqrt{q}$ whenever $\chi$ is primitive.\\

One of the main tools in the proof of Theorem \ref{MainThm} is the following.
\begin{lem}
\label{lem:ExpSum}
If $\psi \mod{m}$ is a primitive Dirichlet character, then
\[
\max_{\theta\in[0,1]} \left| \sum_{n \in \Z} a_n \psi(n) e(n\theta) \right|
	\geq
		\frac{\sqrt{m}}{\phi(m)}
			\bigg| \sum_{(n,m) = 1} a_n \bigg|
\]
for any set of complex numbers $\{a_n\}$ satisfying $\sum |a_n| < \infty$.
\end{lem}

\begin{proof}
Since $\psi$ is primitive, we have
\[
\begin{split}
\sum_{b \mod{m}} \psi(b) \sum_{n \in \Z} a_n \psi(n) e\Big(\frac{bn}{m}\Big)
&=
\sum_{n \in \Z} a_n \psi(n) \sum_{b \mod{m}} \psi(b) e\Big(\frac{bn}{m}\Big) \\
&=
\sum_{n \in \Z} a_n \psi(n) \overline{\psi}(n) \tau(\psi) \\
&=
\tau(\psi) \sum_{(n,m) = 1} a_n .
\end{split}
\]
It follows that
\[
\sqrt{m} \left|\sum_{(n,m)=1} a_n\right|
\leq
\sum_{b \mod{m}} \left| \sum_{n \in \Z} a_n \psi(n) e\Big(\frac{bn}{m}\Big)\right|
\leq
\phi(m) \; \max_{\theta\in[0,1]} \left| \sum_{n \in \Z} a_n \psi(n) e(n\theta) \right|
\]
as claimed.
\end{proof}

The next result was established in our earlier work on odd order character sums \cite{GL}.
\begin{lem}[Lemma 2.2 of \cite{GL}]
\label{lem:2-2}
Let $\{a(n)\}_{n\in \mathbb{Z}}$ be a sequence of complex numbers with $|a(n)|\leq 1$ for all $n$, and let $x\geq 2$ be a real number.
Then
\[
\max_{\theta\in [0,1]}\max_{1\leq N\leq x}\left|\sum_{1\leq |n|\leq N}\frac{a(n)}{n} e(n\theta)\right|
=\max_{\theta\in [0,1]}\left|\sum_{1\leq |n|\leq x}\frac{a(n)}{n} e(n\theta)\right|+O(1).
\]
\vspace*{0.1in}
\end{lem}

Our proof of Theorem \ref{MainThm} will require fixing a character satisfying various nice properties. We isolate the construction of a suitable character in the next lemma.
\begin{lem}
\label{lem:Psi}
For any even integer $g \geq 2$, there exists an odd primitive character $\psi$ of order $g$ and prime conductor.
\end{lem}

\begin{proof}
We begin by generating an appropriate conductor, which we denote $m$.
Since $g$ is even, we have $(g+1,2g) = 1$. Dirichlet's theorem implies that there exists a prime $m \equiv g+1 \mod{2g}$; note that $\frac{m-1}{g}$ is odd, a fact we shall require at the end of the proof. Let $\alpha$ denote a primitive root$\mod{m}$, and define a character $\psi\mod{m}$ by setting
$\psi(\alpha) = e\big(\frac{1}{g}\big)$ and extending by complete multiplicativity. It is clear by construction that $\psi$ is a Dirichlet character of order $g$. Moreover, since $m$ is prime, $\psi$ must be primitive. It thus remains only to check that $\psi$ is odd, a straightforward exercise:
\[
\psi(-1) = \psi(\alpha)^{(m-1)/2} = e\Big(\frac{m-1}{2g}\Big) = -1
\]
since $\frac{m-1}{g}$ is an odd integer.
\end{proof}

The final ingredient we shall require is the construction of an odd primitive quadratic character $\chi\mod{q}$ which satisfies $\chi(p)=1$ for all ``small'' primes $p$. The argument we present below follows the same lines as that given in \cite{Pa}, but is streamlined and uses more standard notation.
\begin{pro}
\label{prop:ChiPointsAt1}
There exist arbitrarily large $q$ and odd primitive quadratic characters $\chi\mod{q}$ such that
$\chi(n) = 1$ for all positive integers $n \leq \frac{1}{2} \log q$.
\end{pro}

\begin{proof}
For each prime $p \geq 3$, there exists a natural number $Q_p$ satisfying
\[
\bigg(\frac{Q_p}{p}\bigg) = \chi_{_{-4}}(p)
\]
where $\big(\frac{\cdot}{n}\big)$ denotes the Jacobi-Legendre symbol modulo $n$, and
$\chi_{_{-4}}$ denotes the nonprincipal character modulo $4$. Let $N$ be a large positive integer.
By the Chinese Remainder Theorem, the system of congruences
\[
\begin{split}
x &\equiv -1 \mod{8} \\
x &\equiv Q_p \mod{p} \qquad  \text{for } 3 \leq p \leq N
\end{split}
\]
has a solution $Q\leq\displaystyle 4 \prod_{p \leq N} p$.
It thus follows from the prime number theorem that
\begin{equation}
\label{eq:Qvsy}
\log Q \leq \sum_{p \leq N} \log p + O(1)\leq 2N
\end{equation}
if $N$ is sufficiently large.

Next, we apply quadratic reciprocity. Since $Q \equiv -1 \mod{4}$, we find
\[
\bigg(\frac{p}{Q}\bigg) = (-1)^{(p-1)/2}\bigg(\frac{Q}{p}\bigg)
= \chi_{_{-4}}(p)^2 = 1
\]
for all odd primes $p \leq N$. Furthermore, since $Q \equiv -1 \mod{8}$, we see that
\[
\bigg(\frac{2}{Q}\bigg) = 1  \qquad \text{and} \qquad
\bigg(\frac{-1}{Q}\bigg) = -1 .
\]
Let $\chi\mod{q}$ be the primitive character which induces $\big(\frac{\cdot}{Q}\big)$. Then $\chi$ is odd and quadratic, and ${\chi(p) = 1}$ for all primes $p \leq N$. It follows that $\chi(n)=1$ for all positive $n\leq N$, whence $q \geq N$. Since $\frac{1}{2} \log q \leq \frac{1}{2}\log Q\leq N$ by \eqref{eq:Qvsy}, we conclude.
\end{proof}

\section{Proof of Theorem \ref{MainThm}}

Fix a large number $Y$. It suffices to construct a character $\chi_g$ of order $g$ and conductor $q_g > Y$ which satisfies
\[
M(\chi_g) \gg_g \sqrt{q_g} \log \log q_g.
\]
We construct such a character in three steps.
First, Lemma \ref{lem:Psi} guarantees the existence of an odd primitive character $\psi_g$ of order $g$ whose conductor $m$ is prime. This character will be fixed throughout the argument (and is independent of our choice of $Y$), so we can write $m \asymp_g 1$. Next, Proposition \ref{prop:ChiPointsAt1} yields an odd primitive quadratic character $\chi$ of conductor $q > Y$ satisfying $\chi(n) = 1$ for all $n \leq \frac{1}{2}\log q$. Finally, let $\chi_g \mod{q_g}$ be the primitive character inducing $\chi\psi_g$. We observe that $\chi_g$ has order $g$, and that $q_g \asymp_g q$. Furthermore, $\chi_g$ is an even character. \\

Using P\'olya's fourier expansion \cite{MV} and the fact that $\chi_g$ is even, we have
\[
\begin{split}
\sum_{n\leq t}\chi_g(n)
&=
	\frac{\tau(\chi_g)}{2\pi i}
	\sum_{1\leq |n|\leq q_g} \frac{\overline{\chi_g}(n)}{n}
		\left(1-e\left(-\frac{nt}{q_g}\right)\right)
		+O_g(\log q) \\
&=
	-\frac{\tau(\chi_g)}{2\pi i}
	\sum_{1\leq |n|\leq q_g} \frac{\overline{\chi_g}(n)}{n}
		e\left(-\frac{nt}{q_g}\right)
		+O_g(\log q).
\end{split}
\]
It follows that
\begin{equation}
\label{eq:3-1}
M(\chi_g)
\geq
\frac{\sqrt{q_g}}{2\pi} \max_{\theta\in [0,1]}\left|\sum_{1\leq |n|\leq q_g}\frac{\chi_g(n)}{n}e(n\theta)\right| + O_g(\log q) .
\end{equation}
Applying Lemma \ref{lem:2-2} yields
\begin{equation}
\label{eq:3-2}
\max_{\theta\in [0,1]}
	\left|\sum_{1\leq |n|\leq q_g}\frac{\chi_g(n)}{n}e(n\theta)\right|
\geq
\max_{\theta\in [0,1]}
	\left|\sum_{1\leq |n| \leq \frac{1}{2}\log q}
		\frac{\chi(n)\psi_g(n)}{n}e(n\theta)\right|+O(1).
\end{equation}
The final step of our argument, an application of Lemma \ref{lem:ExpSum} (see below for details), is to show that
\begin{equation}
\label{eq:3-3}
\max_{\theta\in [0,1]}
	\left|\sum_{1\leq |n| \leq \frac{1}{2}\log q}
		\frac{\chi(n)\psi_g(n)}{n}e(n\theta)\right|
\geq
\frac{\sqrt{m}}{\phi(m)} \sum_{n \leq \frac{1}{2}\log q} \frac{1}{n} .
\end{equation}
Combining equations (\ref{eq:3-1}) -- (\ref{eq:3-3}) gives
\[
\begin{split}
M(\chi_g) &\geq \frac{\sqrt{m}}{2 \pi \phi(m)} \sqrt{q_g} \sum_{n \leq \frac{1}{2}\log q} \frac{1}{n}+ O_g(\sqrt{q_g}) \\
&\gg_g \sqrt{q_g} \log \log q_g
\end{split}
\]
as desired, since $m \asymp_g 1$ and $q \asymp_g q_g$. \\

To conclude the proof of Theorem \ref{MainThm} it remains only to prove the bound (\ref{eq:3-3}). From Lemma \ref{lem:ExpSum} we infer that
\[
\max_{\theta\in [0,1]}
	\left|\sum_{1\leq |n| \leq \frac{1}{2}\log q}
		\frac{\chi(n)\psi_g(n)}{n}e(n\theta)\right|
\geq
\frac{\sqrt{m}}{\phi(m)}
	\left| \sum_{\substack{1 \leq |n| \leq \frac{1}{2}\log q \\ (n,m)=1}}
		\frac{\chi(n)}{n}\right| .
\]
Since $m$ is prime, $\chi(-1) = -1$, and $\chi(n) = 1$ for all $n \leq \frac{1}{2}\log q$, we have
\[
\begin{split}
\frac{\sqrt{m}}{\phi(m)}
	\Bigg| \sum_{\substack{1 \leq |n| \leq \frac{1}{2}\log q \\ (n,m)=1}}
		\frac{\chi(n)}{n} \Bigg|
&=
\frac{2 \sqrt{m}}{\phi(m)}
	\sum_{\substack{n \leq \frac{1}{2}\log q \\ m \dnd n}} \frac{1}{n} \\
&\geq
\frac{2 \sqrt{m}}{\phi(m)} \left(1 - \frac{1}{m}\right)
	\sum_{n \leq \frac{1}{2}\log q} \frac{1}{n} \\
&\geq
\frac{\sqrt{m}}{\phi(m)} \sum_{n \leq \frac{1}{2}\log q} \frac{1}{n} .
\end{split}
\]
We thus obtain the bound (\ref{eq:3-3}), and complete the proof of Theorem \ref{MainThm}.

\noindent\\
\textbf{Acknowledgments.}
We are grateful to Jonathan Bober for pointing out a small error in an earlier draft.


\begin{thebibliography}{DDDD}


\bibitem{Da} H. Davenport,
\emph{Multiplicative number theory},
Graduate Texts in Mathematics, 74, Springer-Verlag, New York, 2000.


\bibitem{Go} L. Goldmakher,
\emph{Multiplicative mimicry and improvements of the P\'{o}lya-Vinogradov inequality},
to appear in Algebra and Number Theory.


\bibitem{GL}  L. Goldmakher and Y. Lamzouri,
\emph{Lower bounds on odd order character sums}, to appear in IMRN.


\bibitem{GS}  A. Granville and K. Soundararajan,
\emph{Large character sums: pretentious characters and the P\'{o}lya-Vinogradov theorem},
J. Amer. Math. Soc. 20 (2007), no. 2, 357-384.



\bibitem{MV}
H. L. Montgomery and R. C. Vaughan,
\emph{Exponential sums with multiplicative coefficients},
Invent. Math. {\bf 43} (1977), 69-82.


\bibitem{Pa}
R. E. A. C. Paley,
\emph{A theorem on characters},
J. London Math. Soc.
{\bf 7} (1932),
28-32.



\end{thebibliography}
\end{document}